\documentclass[11pt]{article}
\usepackage{latexsym}
\usepackage{epsfig,enumerate}
\usepackage{amsmath,amsthm,amssymb}
\usepackage{algorithm2e}
\usepackage{amsbsy,enumerate}
\usepackage{bbm}
\usepackage{hyperref}
\allowdisplaybreaks

\usepackage[usenames]{color}\usepackage{graphicx}
\numberwithin{equation}{section}
\definecolor{brown}{cmyk}{0, 0.72, 1, 0.45}
\definecolor{grey}{gray}{0.5}

\renewcommand{\epsilon}{\varepsilon}

\def\deg{\text{deg}}

\newcounter{rot}%\addtocounter{rot}{1}, \therot

\def\a{\alpha}   \def\D{\Delta}
    
\def\G{\Gamma}

\newtheorem*{conjecture*}{Conjecture}
\newtheorem{theorem}{Theorem}[section]
\newtheorem*{theorem*}{Theorem}
\newtheorem{lemma}[theorem]{Lemma}

\newcommand{\wt}[1]{\widetilde{#1}}
\newcommand{\wtq}{\wt{Q_I}}

\newcommand{\rbrac}[1]{\left(#1\right)}
\newcommand{\sbrac}[1]{\left[ #1\right]}
\newcommand{\cbrac}[1]{\left\{ #1\right\}}

\newcommand{\rfrac}[2]{\left(\frac{#1}{#2}\right)}

\def\E{\mathbb{E}}

\def\Var{\mbox{{\bf Var}}}
\def\P{\mathbb{P}}

\newcommand{\bfo}{\mathbbm{1}}

\newcommand{\eps}{\varepsilon}

\newcommand{\of}[1]{\left( #1 \right) }

\newcommand{\abs}[1]{\left| #1 \right|}

\newcommand{\sqbs}[1]{\left[ #1 \right]}
\newcommand{\braces}[1]{\left\{ #1 \right\}}
\newcommand{\Mean}[1]{\E\sqbs{#1}}

\newcommand{\opoo}{(1+o(1))}

\allowdisplaybreaks[1]
\newcommand{\ix}[2]{I_{#1}^{(#2)}}

\newcommand{\ignore}[1]{}

\newcommand{\beq}[1]{\begin{equation}\label{#1}}
\newcommand{\eeq}{\end{equation}}

\newcommand{\mc}[1]{\mathcal{#1}}

\mathchardef\UrlBreakPenalty=10000
\title{The bipartite $K_{2,2}$-free process and bipartite Ramsey number $b(2, t)$}

\author{{\Large Deepak Bal}\thanks{\url{deepak.bal@montclair.edu} } \\Department of Mathematical Sciences \\ Montclair State University \\ Montclair, NJ 07043 \and {\Large Patrick Bennett}\thanks{\url{patrick.bennett@wmich.edu}, supported in part by Simons Foundation Grant \#426894.} \\Department of Mathematics \\ Western Michigan University \\ Kalamazoo, MI, 49008}
\date{}
\begin{document}
\maketitle

\begin{abstract}
The bipartite Ramsey number $b(s,t)$ is the smallest integer $n$ such that every blue-red edge coloring of $K_{n,n}$ contains either a blue $K_{s,s}$ or a red $K_{t,t}$. In the bipartite $K_{2,2}$-free process, we begin with an empty graph on vertex set $X\cup Y$, $|X|=|Y|=n$. At each step, a random edge from $X\times Y$ is added under the restriction that no $K_{2,2}$ is formed. This step is repeated until no more edges can be added. In this note, we analyze this process and show that the resulting graph witnesses that $b(2,t) =\Omega\of{t^{3/2}/\log t}$, thereby improving the best known lower bound.

\end{abstract}

\section{Introduction}

The \emph{bipartite Ramsey number} $b(s,t)$ is the smallest integer $n$ such that every blue-red edge coloring of $K_{n,n}$ contains either a blue $K_{s,s}$ or a red $K_{t,t}.$ This definition was first introduced by Beineke and Schwenk \cite{BS} in 1976. {We will find it convenient to define the {\em bipartite independence number} of a graph $G \subseteq K_{n, n}$ as the largest value of $t$ such that there are sets of vertices $A$ and $B$ on opposite sides of the bipartition of $K_{n, n}$ such that $|A|=|B|=t$ and $G$ has no edges in $A \times B$. Thus, $b(s, t) > n$ if and only if there exists a $K_{s, s}$-free graph $G\subseteq K_{n, n}$ that has bipartite independence number less than $t$.}
 The best known lower and upper bounds on the diagonal problem, due to Hattingh and Henning \cite{HH} and Conlon \cite{Con} respectively, are
 \[ \frac{\sqrt{2}}{e}t 2^{t/2}   \le  b(t,t) \le (1+o(1))2^{t+1}\log_2 t. \]
As is the case for the ordinary Ramsey number, there is an exponential gap which remains to be closed.

In this note, we are concerned with the simplest nontrivial ``off-diagonal'' case, $b(2,t)$. The best known lower and upper bounds, both due to Caro and Rousseau \cite{CR}, are
\[ \Omega\of{\of{\frac{t}{\log t}}^{3/2}}\le  b(2,t) \le O\of{\of{\frac{t}{\log t}}^2}.\]
The upper bound follows directly from the well known upper bounds on the Zarankiewicz problem. $z(n,s)$ is the largest number of edges in a $K_{s,s}$-free subgraph of $K_{n,n}$. 
%It is well known 
{A theorem of K\"{o}vari, S\'os and Tur\'an \cite{KST} says} that $z(n,s) = O(n^{2-1/s})$. By the  pigeonhole principle, if the number of edges in $K_{n,n}$ exceeds $z(n,2) + z(n,t)$, then in any blue-red coloring of the edges, one color class must exceed its respective Zarankiewicz bound. Thus solving the inequality $n^2 > c_1n^{3/2} + c_2n^{2-1/t}$ for $n$ provides the upper bound.
The proof of the lower bound makes use of the Lov\'{a}sz Local Lemma.

Let $R(G_1, G_2)$ represent the ordinary two color Ramsey number, i.e., the smallest integer $n$ such that every blue-red edge coloring of the edges of $K_n$ contains a blue copy of $G_1$ or a red copy of $G_2$. When $G_1 = K_s, G_2=K_t$, we write $R(s,t)$.  The search for the asymptotics of the off diagonal ordinary Ramsey number $R(3,t)$ 
%(the smallest $n$ such that every red-blue edge coloring of $K_n$ contains a red $K_3$ or a blue $K_t$) 
has a long  and interesting history as laid out by Joel Spencer in \cite{Sp}. Currently, the best known
 results differ only by a constant factor. Shearer \cite{Sh}, improving a result of Ajtai, Koml\'os and Szemer\'edi \cite{AKS}, proved that $R(3,t) \le (1+o(1))t^2/\log t$. Recently, Bohman and Keevash \cite{BK13}, and independently Fiz Pontiveros, Griffiths and Morris \cite{FGM}, proved that $R(3,t) \ge (\frac14 -o(1))t^2/\log t$. The lower bound follows from the analysis of a random process first considered by Erd\H{o}s, Suen and Winkler \cite{ESW} and now commonly referred to as the triangle-free process. This is a stochastic process in which random edges are added to an empty graph one by one under the constraint that no triangles are formed. The authors of \cite{BK13} and \cite{FGM} prove bounds on how long this process lasts and the size of the largest independent set in the resulting graph. 
 
  Of course one can also consider the $H$-free process for any $H$, where edges are randomly added one by one under the constraint that no copy of $H$ is formed. %Probably the first people to study such a process were 
  Ruci\'nski and Wormald \cite{worm} {were among the first to consider such a process, analyzing} 
 % who analyzed 
 the $d$-process, which is the $H$-free process where $H$ is a star on $d+1$ vertices. Further work has been done to analyze the $H$-free process for other families of graphs $H$, mostly when $H$ is a clique or a cycle (see for example Picollelli \cite{mike3, mike2} and Warnke \cite{lutz2, lutz1}). Bohman and Keevash \cite{BK10} have the most general results for the $H$-free process; they analyze the process and bound the independence number of the resulting graph for a large class of graphs $H$ including cycles of any length as well as cliques of any size (but also all strictly 2-balanced graphs), establishing new lower bounds on Ramsey numbers $R(H, K_t)$ where $H$ is any fixed cycle or clique and $t \rightarrow \infty$. Bohman, Picollelli and Mubayi \cite{BMP} studied the $H$-free process for certain hypergraphs $H$, resulting in new lower bounds for the corresponding hypergraph Ramsey numbers. %\Patrick{Should I cite anyone else? There's definitely a few more papers about H-free processes but they aren't as directly relevant. Maybe we should mention the d-process though and give Wormald a shoutout}\deepak{I think what you added is good and more than enough. But if you are sending this to Wormald and want to mention Rucinski Wormald d-process (which it seems like was the first "constrained graph process") then I'm fine with it)}

Inspired by the previous work on $H$-free processes, we study the \emph{bipartite $K_{2,2}$-free process}, a version of the $H$-free process in a large balanced bipartite host graph (as opposed to the standard $H$-free process which uses $K_n$ as a host graph). The process begins with an empty graph $G_0$ on vertex set $X\cup Y$ where $|X|= |Y| = n$. We form the graph $G_{i}$ by adding to $G_{i-1}$ an edge $e_i$ chosen uniformly at random from all pairs of vertices in $X\times Y$ which do not already appear in $G_{i-1}$ and which do not create a copy of $K_{2,2}$. Let $M$ be the random variable representing the number of edges in the final graph produced at the end of the process. Then $G_M$ is $K_{2,2}$-free by construction.  The main contribution of this paper is to prove that with high probability, $G_M$ has bipartite independence number at most $Cn^{2/3}\log^{2/3}n$ for some constant $C$.  

\begin{theorem}\label{thm:main}
With high probability, the graph produced at the end of the bipartite $K_{2,2}$-free process has bipartite independence number $O(n^{2/3}\log^{2/3}n)$. Thus $$b(2,t) = \Omega\of{\frac{t^{3/2}}{\log t}}.$$
\end{theorem}
%We note that this provides an asymptotic $\sqrt{\log t}$ improvement over the lower bound provided in \cite{CR}. 

 %One of the goals of this paper is to demonstrate how the general result of \cite{BB} can be used to drastically reduce the length (complexity?) of proofs involving the analysis of constrained random graph processes.

\section{Proof of Theorem \ref{thm:main}}\label{sec:proof}

 Bennett and Bohman \cite{BB} proved a general result which we will find useful. We build our proof ``on top of" the proof of Theorem 1.1 in \cite{BB}, in the sense that we will use not only the statement of that theorem but also some other facts that are established in its proof. In particular, Theorem 1.1 in \cite{BB} is proved by establishing dynamic concentration of a family of random variables. In our proof, we will use (without further justification) the fact that these variables are dynamically concentrated. {One of the goals of this paper is to demonstrate the utility of \cite{BB} as a ``black box" that takes care of a lot of the work of analyzing these processes, making for shorter proofs.}

In Section \ref{sec:blackbox} we will summarize the relevant results from \cite{BB}. In Section \ref{sec:bad-dad} we prove a few lemmas which do not follow directly from the results in \cite{BB}, namely, bounds on the maximum degree and on the maximum density of subsets.  Finally, in Section \ref{sec:ind} we bound the bipartite independence number of the graph produced by the process by using similar proof techniques to those used in \cite{BK10} for analyzing the $H$-free process when $H$ is a cycle.

\subsection{The Black Box}\label{sec:blackbox}
In this section, we summarize the results from Bennett and Bohman \cite{BB} which we will utilize. Let $ {\mathcal H} $ be a hypergraph on vertex
set $V$ (i.e. ${\mathcal H}$ is a collection of subsets of $V$ and the sets in this
collection are the {\em edges} of $\mc{H}$).  An {\em independent set} in $\mc{H}$ 
is a set $ I \subseteq V $ such that
$I$ contains no edge of $ {\mathcal H} $.  The {\em random greedy independent set process} (or just the {\em independent process}) forms a maximal
independent set in $ {\mathcal H}$ by iteratively choosing vertices at random to be 
in the independent set.  To be precise, we begin with $ {\mathcal H}(0) = {\mathcal H}$, $ V(0) = V$ 
and $ I(0) = \emptyset $.  Given independent set $I(i)$ and 
hypergraph $ \mc{H}(i) $ on vertex set $V(i)$, a vertex $v \in V(i)$ is chosen uniformly at random
and added to $ I(i)$ to form $ I(i+1) $.  The new vertex set $V(i+1)$ and new hypergraph $\mc{H}(i+1)$ are formed by 
\begin{enumerate}
\item removing $v$ from every edge in $ \mc{H}(i) $ that contains $v$ (so these edges become smaller edges),
\item deleting $v$ from $V(i)$, and
\item deleting from $V(i)$ every vertex that is in a singleton edge (and any such edge containing a deleted vertex is removed).
\end{enumerate}

Define the {\em degree} of a set $ A \subseteq V$ to be the number of edges of 
$ {\mathcal H} $ that contain $A$.
For $ a = 2, \dots, r-1 $ we define $ \Delta_a( {\mathcal H}) $
to be the maximum degree of $A$ over $A \in \binom{V}{a} $.  We also
define the {\em $b$-codegree} of a pair of distinct vertices $ v,v'$ to be 
the number of pairs of edges $e,e' \in {\mathcal H} $ such that $ v \in e \setminus e', v' \in e' \setminus e$ 
and $ |e \cap e'|=b $.  We let $ \G_b ( \mc{H}) $ be the maximum 
$b$-codegree of $ {\mathcal H} $.
\begin{theorem}[Theorem 1.1 in \cite{BB}]
\label{thm:BB}
Let $r$ and $ \epsilon > 0 $ be fixed.  Let 
$ {\mathcal H} $ be a $ r$-uniform, $D$-regular hypergraph on $N$ 
vertices such that $ D > N^{\epsilon} $.  If
\begin{equation}
\label{eq:degcond}
 \Delta_\ell( {\mathcal H} ) < D^{ \frac{r-\ell}{r-1} - \epsilon} \ \ \ \text{ for }
\ell = 2, \dots, r-1 
\end{equation}
and $ \Gamma_{r-1}( {\mathcal H} ) < D^{1- \epsilon} $
then the random greedy independent set
algorithm produces an independent set $I$ in $ {\mathcal H} $ with
\begin{equation}
\label{eq:lowerbound}
|I| = \Omega_{r, \eps}\left( N \cdot \left( \frac{\log N}{ D} \right)^{\frac{1}{r-1}} \right)
\end{equation}
with probability $ 1 - \exp\left\{ - N^{\Omega_{r, \eps}(1)} \right\} $.
\end{theorem}
 The $K_{2,2}$-free process in the host graph $K_{n, n}$ is an instance of the independent process in a hypergraph $\mc{H}_{K_{2,2}} = \mc{H}_{K_{2,2}}(n)$, where each vertex in $\mc{H}_{K_{2,2}}$ corresponds to an edge in $K_{n, n}$ and edges in $\mc{H}_{K_{2,2}}$ correspond to sets of edges in $K_{n, n}$ that form copies of $K_{2, 2}$. Thus $\mc{H}_{K_{2,2}}$ is a $4$-uniform, $(n-1)^2$-regular hypergraph on $n^2$ vertices. In the notation of 
 %\cite{BB}
 Theorem \ref{thm:BB}, we have $r=4$, $D=(n-1)^2$, $N=n^2$. Also, we have $\D_2(\mc{H}_{K_{2,2}}) = O(n)$, $\D_3(\mc{H}_{K_{2,2}}) = O(1)$ and $\G_3(\mc{H}_{K_{2,2}})=0$, so the conditions of Theorem \ref{thm:BB} are met with $r=4$ and any $0 < \eps < 1/3 $. Thus, Theorem \ref{thm:BB} tells us that the bipartite $K_{2,2}$-free process admits $\Omega\of{ n^{4/3} \log^{1/3} n } $ many edges. For the rest of the paper, we will almost exclusively use language referring to the $K_{2,2}$-free process as opposed to language referring to the independent process on $\mc{H}_{K_{2,2}}$.
 
Now we state the dynamic concentration results that are established in \cite{BB} as part of the proof of Theorem \ref{thm:BB}. As in \cite{BB} we define the {\em scaled time parameter} 
$$t=t(i):= \frac{D^\frac{1}{r-1}}{N} \cdot i = \frac{(n-1)^{2/3}}{n^2} \cdot i  = \frac{i}{n^{4/3}}\of{1+O\of{\frac{1}{n}}} ,$$
and
\[
q=q(t):=e^{-t^3}.
\]
We let $Q = Q(i)$ be the set of {\em open} edges at step $i${, that is,} the edges that could be chosen without creating a $K_{2,2}$. For each open edge $e \in Q(i)$ and we define $d_2(e)$ to be the number of copies of $K_{2,2}$ that contain $e$, contain one additional open edge, and two chosen edges. Roughly speaking $d_2(e)$ is the number of open edges that, if chosen, would close $e$.
%is the number of ``ways" that $e$ might be closed at a given step. 
%Equivalently, we may think of $d_2(e)$ as the number of open edges which would become closed if $e$ were chosen \pb{(alternatively, the number of open edges that, if chosen, would close $e$)}.  
In the proof of their theorem, Bennett and Bohman establish dynamic concentration of the random variables $Q$ and $d_2(e)$ (see equations (8) and (9) in \cite{BB}). In particular, their proof implies that there exists a positive constant $\eps$ such that w.h.p. 
\begin{align} 
\label{eq:open}
| Q | & \in \of{1\pm n^{-10\eps^3}}n^2q   \\
\label{eq:d2}
d_2(e)  & \in \of{1\pm n^{-10\eps^3}}3n^{2/3}t^2q  \;\;\; \mbox{   for all } e\in Q
\end{align}
for all $i \le  i_{max} :=\eps n^{4/3} \log^{1/3} n$. Note that lines \eqref{eq:open} and \eqref{eq:d2} significantly simplify the equations in \cite{BB} which actually refer to several constants. We satisfy \eqref{eq:open} and \eqref{eq:d2} by choosing $\eps$ to be sufficiently small (and we will continue to assume $\eps>0$ is sufficiently small throughout the paper). Indeed, we will write many inequalities that only hold under the assumption that $n$ is sufficiently large and $\eps>0$ is sufficiently small. 

\subsection{High degrees and dense subgraphs}\label{sec:bad-dad}
Let $G_i$ be the $K_{2, 2}$-free graph at step $i$.  In this section we show that w.h.p. $G_i$ does not have any vertices of degree too high, nor subgraphs that are too dense. More specifically, $G_i$ resembles (at least in these aspects) a binomial random bipartite graph with edge probability $\frac{i}{n^2}$. Let $\mc{E}_i$ be the event that \eqref{eq:open}, \eqref{eq:d2} hold for all steps up to and including step $i$. 

\begin{lemma}\label{lem:spec-edges}
For any set of edges $F\subseteq E(K_{n,n})$, we have  $$\P\sbrac{\mc{E}_{i_{max}} \mbox{ and } F \subseteq E(G_{i_{max})}} \le n^{(-2/3 + 2\eps^3)|F|}.$$
\end{lemma}

\begin{proof}
 The probability that all edges of $F$ are chosen is at most the number of ways to specify which steps these edges will be chosen multiplied by the probability of choosing the prescribed edges in the specified steps. If we know that $\mc{E}_{i_{max}}$ holds,then for any $i \le i_{max}$
 \[Q(i)\ge (1-n^{-10\eps^3 })n^2q\of{\frac{i_{max}}{n^{4/3}}} =(1-n^{-10\eps^3 })n^2q(\eps \log^{1/3} n) \ge (1+o(1))n^{2-\eps^3} \] and so the probability of choosing a particular edge on a particular step is at most $\frac{1}{Q} < \frac{1+o(1)}{ n^{2-\eps^3}}$ (conditional on the history of the process). So we have that the probability that all edges of $F$ are chosen is at most 
 $$(\eps n^{4/3} \log^{1/3}n)_{|F|} \cdot \rfrac{1+o(1)}{ n^{2-\eps^3}}^{|F|} < n^{(-2/3 + 2\eps^3)|F|}.$$
\end{proof}

\begin{lemma}\label{cor:deg}
W.h.p. we have
\begin{equation}\label{eq:deg}
{\deg(v) \le n^{1/3 + 3 \eps^3}  \quad \mbox{for all vertices $v$}.}
\end{equation}

\end{lemma}

\begin{proof}
By Lemma \ref{lem:spec-edges}, we have that the expected number of vertices of degree at least $n^{1/3 + 3 \eps^3}$ at step $i_{max}$ (and hence at any step in the process) is at most 
\begin{align*}
2n \cdot \binom{n}{n^{1/3 + 3 \eps^3}} \cdot n^{(-2/3 + 2\eps^3)n^{1/3 + 3 \eps^3}}&\le    2n\cdot\exp\braces{n^{1/3 + 3\eps^3}\of{\log\of{\frac{ne}{n^{1/3 + 3\eps^3}}}  + (-2/3 + 2\eps^3)\log n   }    }\\
&=2n\cdot \exp\braces{n^{1/3 + 3\eps^3}\of{ \of{2/3  -3\eps^3  -2/3 + 2\eps^3}\log n     +1}  }  \\ 
&= 2n \cdot \exp \cbrac{- \opoo\eps^3n^{1/3 + 3 \eps^3} \log n}= o(1).
\end{align*}
The result follows from Markov's inequality.
\end{proof}

\begin{lemma}\label{cor:dense}
W.h.p. we have
\begin{equation}\label{eq:dense}
{e(A, B) \le 2 \eps^{-3} \max\{a+b, abn^{-2/3+3\eps^3} \}\quad  \mbox{for all $A \subseteq X$ and $B \subseteq Y$ with $|A|=a, |B|=b$.}}
\end{equation}
\end{lemma}

\begin{proof}
Set $h=h(a,b)= 2 \eps^{-3} \max\{a+b, abn^{-2/3+3\eps^3} \}$. Then using Lemma \ref{lem:spec-edges},  the probability that there exist sets $A\subseteq X,  B\subseteq Y$ (of size $a$ and $b$ respectively)  with $e(A,B)\ge h(a,b)$ at step $i_{max}$ (and hence at any step) is at most
\begin{align*}
\sum_{1\le a, b\le n} n^{a+b}  \binom{ab}{h} n^{(-2/3 + 2\eps^3) h}
&\le \sum_{1\le a, b\le n} \exp\cbrac{\sbrac{a+b + \rbrac{-\frac23 +2\eps^3}h} \log n + h \log\rfrac{abe}{h} }\\
&\le \sum_{1\le a, b\le n} \exp\cbrac{  \rbrac{\frac{5}{2}\eps^3-\frac23 }h \log n + h \log\rfrac{abe}{abn^{-2/3+3\eps^3}} }\\
&\le \sum_{1\le a, b\le n} \exp\cbrac{ (1+o(1)) \rbrac{-\frac{1}{2}\eps^3}h \log n  }\\
& \le \sum_{1\le a, b\le n} \exp\cbrac{ -(1+o(1))(a+b) \log n  } = o(1).
\end{align*}
On the second and fourth line we have used the fact that $h \ge 2 \eps^{-3}(a+b)$ and {on the second line we used $h \ge 2 \eps^{-3}abn^{-2/3+3\eps^3}$}. 
%On the fourth line we also used \eqref{eq:ke}.
\end{proof}

\subsection{Bipartite Independence Number} \label{sec:ind}

Let $I_X \subseteq X$ and $I_Y \subseteq Y$ with \[|I_X| = |I_Y| = \a := 2 \eps^{-1} n^{2/3} \log^{2/3} n. \] We would like to show that the number of open pairs in $I:=I_X \times I_Y$ remains significant throughout the process so that an edge will land in one of these open pairs with high probability. 
Define $Q_I = Q_I(i)$ to be the number of open pairs in $I$ at step $i$. We would like to show that $Q_I \approx \alpha^2 q$ {throughout the entire process, for every choice of $I$; i.e. that the density of open pairs in each $I$ is approximately the same as the global density of open pairs. }

We will track $Q_I$ by writing $Q_I = \wt{Q_I} - A_I$, where $A_I$ represents the effect of ``large" one-step changes. Formally, define
\[
 A_I(i) = \sum_{j \le i} |\D Q_I(j)| \cdot \bfo_{|\D Q_I(j)|> n^{2/3 - 12\eps^{3} }}.
\]
Our motivation for defining $A_I$ is to ensure that the one-step change of $\wtq$ is not too large. We will then be able to apply
% to be useful when we apply 
a martingale inequality. We will establish dynamic concentration for $\wtq$, and just a crude bound for $A_I$ which can be regarded as an error term. 

\subsubsection{Bounding \texorpdfstring{$A_I$}{AI}}\label{sec:boundA}
In this section, we prove the following lemma which provides an upper bound on $A_I$. 
\begin{lemma}
With high probability, for every $I$ and every $i\le i_{max}$, we have $A_I(i) \le n^{1+45\eps^{3}}$.
\end{lemma}

\begin{proof}
 Assume that $\mc{E}_{i_{max}}$ {and \eqref{eq:deg}, \eqref{eq:dense} hold}. Under these assumptions, we will show that for every $I$ and every $i\le i_{max}$, $A_I(i) \le n^{1+45\eps^{3}}$. 
 %Assume $\mc{E}_{i_{max}}$ holds. 
Fix an $I = I_X \times I_Y$.  
 Let $I_X^{(1)}$ be the set of vertices with at least 
$n^{1/3-16\eps^{3}}$ 
 neighbors in $I_X$, and similarly define $I_Y^{(1)}$. 
Let $I_X^{(2)}$ be the set of vertices with at least $n^{ 1/3-16\eps^{3}}$ neighbors in $I_X^{(1)}$, and similarly define $I_Y^{(2)}$. 
 We first claim that in $G_{i_{max}}$,
 \begin{equation}\label{eq:3types}
e(I_Y^{(1)} , I_X^{(1)}) +e(I_X ,I_Y^{(2)})+ e(I_X^{(2)} , I_Y) = O\of{n^{1/3 + 37\eps^{3} } }.
 \end{equation}
 
  Recall that $|I_X| = \a = 2 \eps^{-1} n^{2/3}\log^{2/3}n$.  Then by {\eqref{eq:dense}} we have
 \[ |\ix{X}{1}|\cdot n^{1/3 - 16\eps^{3}} \le e(\ix{X}{1}, I_X) \le 2 \eps^{-3}\max\braces{\a + |\ix{X}{1}|,  \a\cdot |\ix{X}{1}|\cdot n^{-2/3 + 3\eps^3}  }.\]
 If the {maximum above were achieved by the second argument}, we would have a contradiction since $n^{3\eps^3}\log^{2/3}n \ll n^{1/3 - 16\eps^{3}} $. Thus we have $| \ix{X}{1}|\cdot n^{1/3 - 16\eps^{3}} \le 2 \eps^{-3}(\a + |\ix{X}{1}|)$ and so %$|\ix{X}{1}|(n^{1/3-24\eps^{3}} - \k_e) \le \k_e \a$ which implies that 
 $ |\ix{X}{1}| \le 2 \eps^{-3} \a / (n^{1/3-16\eps^{3}} - 2 \eps^{-3}) $. Thus (since $\ix{Y}{1}$ is similar) we have
\begin{equation}\label{eq:I1}
 |I_X^{(1)}|, |I_Y^{(1)}| < n^{1/3 + 17\eps^{3} }.
\end{equation} 

%< \a n^{ -1/3 + 24\eps^{3} + \eps^3} 
 %\deepak{Do we only need $\eps^3$ here, not $2\eps^3$?}\Patrick{I think it's fine. We put an extra $n^{\eps^3}$ in place of a constant and a power of log }

 Let $I_X^{(2)}$ be the set of vertices with at least $n^{ 1/3-16\eps^{3}}$ neighbors in $I_X^{(1)}$, and similarly define $I_Y^{(2)}$. By {\eqref{eq:dense}}, we have
 \[|\ix{X}{2}|\cdot n^{1/3-16\eps^{3}}\le  e(\ix{X}{2}, \ix{X}{1}) \le 2 \eps^{-3}\max\braces{|\ix{X}{2}| + |\ix{X}{1}|,  |\ix{X}{2}|\cdot |\ix{X}{1}|\cdot n^{-2/3 + 3\eps^3}  }. \]
 The first {argument} must be the maximum otherwise we get a contradiction.
Thus using \eqref{eq:I1}, we have
$|\ix{X}{2}|\cdot n^{1/3-16\eps^{3}}\le  2 \eps^{-3}(n^{1/3 +17\eps^{3} } + |\ix{X}{2}|) $ and so rearranging (and since $\ix{Y}{2}$ is similar) we get
\begin{equation}\label{eq:I2}
 |I_X^{(2)}|, |I_Y^{(2)}| < n^{34\eps^{3}}.
\end{equation} 

Note that by  \eqref{eq:dense}  and \eqref{eq:I1} we have w.h.p.,
\[e(I_Y^{(1)} , I_X^{(1)}) \le 2 \eps^{-3}\max\braces{|I_Y^{(1)}| +|I_X^{(1)}|\, ,\, |I_Y^{(1)} ||I_X^{(1)}|n^{-2/3 + 3\eps^3}  }  =O(n^{1/3 + 17\eps^{3} })\]
since the maximum is the first argument. 
By \eqref{eq:deg} and \eqref{eq:I2} we have $e(I_X ,I_Y^{(2)})+ e(I_X^{(2)} , I_Y) \le O\of{n^{34\eps^{3}}\cdot n^{1/3 + 3\eps^3}} = O\of{n^{1/3 + 37\eps^3}}$ w.h.p.. Thus we have proved \eqref{eq:3types}.
 
 In order to have $|\Delta Q_I(i)|> n^{2/3- 12\eps^{3}}$ the edge $(x_i, y_i)$ chosen at step $i$ must be in one of the following three sets: $I_Y^{(1)} \times I_X^{(1)}$, $I_X \times I_Y^{(2)}$, or $I_X^{(2)} \times I_Y$. Indeed, suppose $(x_i, y_i)$ is not in any of the three sets. Then 
\begin{equation}\label{eq:DQbound}
 |\D Q_I| < d_I(x_i)d_I(y_i) + d^{(2)}_I(y_i)\bfo_{x_i \in I_X} + d^{(2)}_I(x_i)\bfo_{y_i \in I_Y} 
\end{equation}
and since $(x_i, y_i) \notin I_X \times I_Y^{(2)}$, we can bound $d^{(2)}_I(y_i)\bfo_{x_i \in I_X}$ as follows. Either $x_i \notin I_X$ in which case we get 0, or $y_i \notin I_Y^{(2)}$ in which case $y_i$ has at most $n^{1/3 -16\eps^{3}}$ neighbors in $I_Y^{(1)}$ which may have as many as $n^{1/3 + 3\eps^3}$ neighbors in $I_Y$, and $y_i$ has at most $n^{1/3 + 3\eps^3}$  additional neighbors that all have at most $n^{1/3-16\eps^{3}}$ neighbors in $I_Y$. Thus 
\[
d^{(2)}_I(y_i)\bfo_{x_i \in I_X} \le n^{1/3-16\eps^{3}} \cdot n^{1/3 + 3\eps^3} + n^{1/3 + 3\eps^3} \cdot n^{1/3-16\eps^{3}} = O(n^{2/3 -13\eps^{3} }) .
\]
We have the same bound on $d_I^{(2)}(x_i)\bfo_{y_i \in I_Y}$ by symmetry. To bound $d_I(x_i)d_I(y_i)$, note that at least one of $x_i\not \in I_Y^{(1)}$ or  $y_i\not \in I_X^{(1)}$ holds. Thus $d_I(x_i)d_I(y_i) \le n^{1/3 + 3\eps^3}\cdot n^{1/3-16\eps^3} = O(n^{2/3- 13\eps^3})$. Thus we have $|\D Q_I|<O(n^{2/3 -13\eps^{3} }) <  n^{2/3 - 12\eps^{3}}$. 

Thus using \eqref{eq:3types}, we see that the number of steps with $|\D Q_I| \ge n^{2/3-12\eps^{3}}$ can be bounded  by $e(I_Y^{(1)} , I_X^{(1)}) +e(I_X ,I_Y^{(2)})+ e(I_X^{(2)} , I_Y) = O\of{n^{1/3 + 37\eps^{3}}}$. 
%Thus the number of steps such that $|\D Q_I| \ge n^{2/3-\d}$ is at most $O\of{n^{1/3 + 48\eps^{3} + 5\eps^3}}$, and 
So for all $i \le \eps n^{4/3} \log^{1/3} n$ and all $I$, we have 
\[
A_I(i) \le O\of{n^{2/3 + 6\eps^3} n^{1/3 + 37\eps^{3} }}< n^{1+45\eps^{3} }
\]
where we use the fact that the largest possible value of $|\D Q_I(i)|$ is the square of the maximum degree, $n^{1/3 + 3\eps^3}$.

\end{proof}

\subsubsection{Dynamic concentration of \texorpdfstring{$\wtq(i)$}{QItilde}}\label{sec:boundQ}

Define $\mc{E}_i' \subseteq \mc{E}_i$ as the event that $\mc{E}_i$ holds, and that for all $I$ and for all $j \le i$ we have
\[
A_I(j) < n^{1+45\eps^{3} }
\] 
(which we proved holds w.h.p. in the last section) as well as 
\begin{equation}\label{eq:Qtraj}
\wtq(j) \in \alpha^2q\rbrac{t(j)} \pm f(t(j))
\end{equation}
where 
\begin{equation*}
f(t) =  n^{4/3 - 5\eps^3  } e^{t^3+t}.
\end{equation*}
Note that since $t \le \eps \log^{1/3} n$ we have 
$f(t) \le n^{4/3 -  3\eps^3 }.$
We now define $\wtq^+$ and $\wtq^-$ as
\begin{equation*}
\wtq^\pm (i):= \begin{cases} 
& \wtq(i) - \alpha^2q(t) \mp f(t) \;\;\; \mbox{ if $\mc{E}_{i-1}'$ holds}\\
& \wtq^\pm (i-1) \;\;\; \mbox{ otherwise}.
\end{cases}
\end{equation*}
 We will show that $\wtq^+$ is a supermartingale.
 %for a suitable function $f$. 
 Since $\D \wtq^+(i)=0$ if $\mc{E}_{i}'$ fails to hold, we will assume $\mc{E}_{i}'$ holds. 

Recall that $d_2(uv)$ is the number of potential copies of $K_{2,2}$ containing the edge $uv$ in which two edges of the $K_{2,2}$ other than $uv$ are in $G(i)$ and the last pair is open. By \eqref{eq:open} and \eqref{eq:d2}, we know that  $Q \in (1\pm n^{-10\eps^3 })n^2q$  and  $d_2(u,v) \in (1\pm n^{-10\eps^3 })3n^{2/3}t^2q$ where we recall that $q = q(t) = e^{-t^3}$. We first calculate
\[\Mean{\D Q_I | \mc{F}_i} = -\frac{1}{Q}\sum_{uv \in Q_I}d_2(uv),\]
and since by \eqref{eq:I1} and \eqref{eq:I2} we have
\begin{align*}
\Mean{\D A_I | \mc{F}_i} &\le n^{2/3 + 6\eps^3} \cdot \frac{|I_Y^{(1)} \times I_X^{(1)}|+|I_X \times I_Y^{(2)}|+ |I_X^{(2)} \times I_Y|}{Q}\\
& \le n^{2/3 + 6\eps^3} \cdot \frac{O\of{n^{2/3 + 35\eps^3}}}{(1+o(1))n^{2-\eps^3}} \le n^{-2/3 + 43\eps^{3}}
\end{align*}
we see that 
\begin{equation}\label{eq:qtildechange}
\Mean{\D \wtq | \mc{F}_i} = -\frac{1}{Q}\sum_{uv \in Q_I}d_2(uv) + O\of{n^{-2/3 + 43\eps^{3}}}.
\end{equation}
{Now we use Taylor's theorem to bound the one-step change of $\a^2 q(t) + f(t)$, the deterministic terms in $\wtq^+$. We have
\begin{align}
&\a^2 q\rbrac{t+\frac{1}{n^{4/3}}} + f\rbrac{t+\frac{1}{n^{4/3}}} - \a^2 q(t) - f(t) \nonumber\\
& = \rbrac{\a^2 q'(t) + f'(t)} \frac{1}{n^{4/3}} + O\rbrac{\rbrac{\a^2 q''(t) + f''(t)} \frac{1}{n^{8/3}}} \nonumber\\
& = \rbrac{-\a^2 3t^2q + f'} n^{-4/3} + O\rbrac{n^{-4/3 +2\eps^3}}\label{eq:detchange}
\end{align}
where on the second line when we write $O\rbrac{\rbrac{\a^2 q''(t) + f''(t)} \frac{1}{n^{8/3}}}$ we mean an absolute bound  on the function inside the big-O that holds holds for all $t \le \eps \log^{1/3}n$. Considering the particular functions $q(t), f(t)$ we arrive at the big-O term on the last line. 
}

Now we will do the supermartingale calculation for $\wtq^+$. Throughout the following, keep in mind that $\a^2 3t^2 q n^{-4/3} \le n^{2\eps^3}$, and that \[n^{-5\eps^3 } \le\frac{f}{\a^2q}  \le n^{-2\eps^3  }.\] 

Thus, {using \eqref{eq:qtildechange}, \eqref{eq:detchange}, and the fact that we have dynamic concentration in the event $\mc{E}_i'$} we have that
 
\begin{align*}
\Mean{\D \wtq^+ | \mc{F}_i} &\le -\frac{1}{(1+n^{-10\eps^3 })n^2q}\of{\a^2q - n^{1+45\eps^{3}} - f}\of{(1-n^{-10\eps^3 })3n^{2/3}t^2q} \\
& \hspace{3cm} + \alpha^2 3t^2 q n^{-4/3} - n^{-4/3}f'+O\of{n^{-2/3 + 43\eps^{3} } + {n^{-4/3 +2\eps^3}}}\\
& =  \alpha^2 3t^2 q n^{-4/3} \sbrac{-\frac{\of{1 - \frac{n^{1+45\eps^{3} }+ f}{\a^2q}}\of{1 - n^{-10\eps^3 }}}{\of{1 + n^{-10\eps^3 }}}+1}- n^{-4/3}f'+O\of{n^{-2/3 + 43\eps^{3} }}\\
& \le \alpha^2 3t^2 q n^{-4/3} \sbrac{\frac{f}{\a^2q}+ 2n^{-10\eps^3 }}- n^{-4/3}f'+O\of{n^{-10\eps^3}}\\
&\le n^{-4/3}\sbrac{3t^2f - f'} +O\of{n^{-8\eps^3}} \le 0.
\end{align*} 
% In the second line we used that $n^{-8/3}f'' \ll n^{-2/3 + 48\eps^{3} + 11\eps^3}$ to simplify the big-O term. 
In the third line, 
we have used the geometric series expansion $1/(1+n^{-10\eps^3 }) = 1 - n^{-10\eps^3 } + O(n^{-2 0\eps^3 })$ to  calculate  
\begin{align*}\frac{\of{1 - \frac{n^{1+45\eps^{3} }+ f}{\a^2q}}\of{1 - n^{-10\eps^3 }}}{\of{1 + n^{-10\eps^3 }}} &= 1 - \frac{n^{1+45\eps^{3}} + f}{\a^2q} - n^{-10\eps^3 } - n^{-10\eps^3 } + O\of{n^{-20\eps^3 } + \frac{f}{\a^2q}n^{-10\eps^3 }  }  \\
&  = 1-\frac{f}{\a^2q} - 2n^{-10\eps^3 } +O\of{n^{-12\eps^3}}.
\end{align*}
 where we used that $\frac{n^{1+45\eps^{3}}}{\a^2q} \ll n^{-20\eps^3 }$.
% 
% 
%\begin{align*}
%\Mean{\D \wtq^+ | \mc{F}_i} &\le -\frac{1}{n^2q + n^{2-\g}}\of{\a^2q - n^{1+60\eps^{3}} - f}\of{3n^{2/3}t^2q - n^{2/3-\g}} \\
%& \hspace{6.5cm} + \alpha^2 3t^2 q n^{-4/3} - n^{-4/3}f'+O\of{n^{-2/3 + 48\eps^{3} + 12\eps^3}}\\
%& =  \alpha^2 3t^2 q n^{-4/3} \sbrac{-\frac{\of{1 - \frac{n^{1+60\eps^{3} }+ f_I}{\a^2q}}\of{1 - \frac{n^{2/3-\g}}{3n^{2/3}t^2q}}}{1 + \frac{n^{2-\g}}{n^2 q }}+1}- n^{-4/3}f_I'+O\of{n^{-2/3 + 48\eps^{3} + 12\eps^3}}\\
%& \le \alpha^2 3t^2 q n^{-4/3} \sbrac{\frac{f_I}{\a^2q}+ \frac{n^{-\g}}{3t^2q}   + \frac{n^{-\g}}{ q }}- n^{-4/3}f_I'+O\of{n^{-\frac54 \g + 2\eps^3}}\\
%&\le n^{-4/3}\sbrac{3t^2f_I - f_I'} +O\of{n^{-\frac54 \g + 2\eps^3}} \le 0
%\end{align*}
Thus $\wtq^+$ is a supermartingale.  

We will use the following martingale inequality found in \cite{Fre}.

\begin{lemma}[Freedman] \label{lem:Freedman}
Let $Y(i)$ be a supermartingale, with $\Delta Y(i) \leq C$ for all $i$, and $V(i) :=\displaystyle \sum_{k \le i} Var[ \Delta Y(k)| \mathcal{F}_{k}]$  Then
$$ \P\left[\exists i: V(i) \le v, Y(i) - Y(0) \geq \lambda \right] \leq \displaystyle \exp\left(-\frac{\lambda^2}{2(v+C\lambda) }\right).$$ \end{lemma}

In order to apply Lemma \ref{lem:Freedman}, we must bound the variance. 

\begin{align}
&\Var\sqbs{\D \wtq | \mc{F}_i} \le \Mean{\of{\D \wtq}^2| \mc{F}_i} \le n^{2/3-12\eps^{3}} \cdot \Mean{\abs{\D \wtq}\,|\, \mc{F}_i} \le n^{2/3-12\eps^{3}}. \label{eq:var}
\end{align}
So by Lemma \ref{lem:Freedman} using $v=n^{2 - 12\eps^{3}} \log n$, $C=n^{2/3-12\eps^{3}}$ and $\lambda = - \wtq^+(0) = f(0)= n^{4/3 - 5\eps^3}$ we see that the probability that $\wtq^+(i_{max})>0$ is at most
\begin{align*} \P\left[\exists i: V(i) \le v, \wtq^+(i) - \wtq^+(0) \geq f(0) \right] &\leq \displaystyle \exp\cbrac{-\frac{f(0)^2}{2(n^{2 - 12\eps^{3}} \log n+f(0) n^{2/3 - 12\eps^{3}} ) }}\\ 
&\le \exp \cbrac{-\Omega\of{n^{2/3  {+\eps^3 }}}}
\end{align*}
which is small enough to overcome a union bound over all choices of $I$ since the number of such choices is 
\begin{equation}\label{eq:Ichoices}
\binom{n}{\a}^2 \le \exp\cbrac{2 \a \log \of{\frac{ne}{\a}} } < \exp\cbrac{ 2 \eps^{-1} n^{2/3} \log^{5/3} n}.
\end{equation}

In a completely analogous fashion, we may prove that $\wtq^-$ remains non-negative until time $i_{max}$ for every $I$, w.h.p. Thus $\mc{E}_{i_{max}}'$ holds w.h.p. for all $i\le i_{max}$.

\subsubsection{Final bound on bipartite independence number}

We have shown that the event $\mc{E}_{i_{max}}'$ holds w.h.p., in which case $Q = n^2q (1+o(1))$ and for all $I$,  $Q_I = \a^2 q(1+o(1))$ for all $i \le i_{max}$. The probability of choosing an edge in $I$ at any given step is $(1+o(1))\frac{\a^2q}{n^2q} \ge {3 \eps^{-2}} n^{-2/3} \log^{4/3} n.$  So the probability that no edge from $I$ is  chosen in the first $i_{max}$ steps is at most 
\[
\of{1-{3 \eps^{-2}} n^{-2/3} \log^{4/3} n}^{\eps n^{4/3} \log^{1/3} n} \le \exp\cbrac{-{3 \eps^{-1}} n^{2/3} \log^{5/3} n}
\]
which is small enough to overcome a union bound over \eqref{eq:Ichoices} many choices. 
%by \eqref{eq:ka}.

\section{Conclusion}
In \cite{CR}, Caro and Rousseau write ``our knowledge of $b(2, n)$ closely parallels that of $R(C_4, K_n)$.'' In this paper, we have furthered this parallel, but the most intriguing open question (as mentioned in \cite{CR}) is to prove or disprove that $b(2, t) = o(t^{2-\eps})$ for some $\eps > 0$.  Erd\H{o}s famously conjectured that $R(C_4, K_n) = o(n^{2-\eps})$. It may be the case that improving the upper bound on $b(2,t)$ is a simpler task than improving that of $R(C_4, K_n)$. Another direction would be to extend our technique to find new lower bounds on $b(s, t)$ with $s$ fixed and $t \to \infty$ or on the bipartite Ramsey number of a fixed cycle versus a large bipartite clique. The techniques used in this paper could be applied here, but we have opted not to pursue this for the sake of brevity. Most likely the techniques in \cite{mike2, lutz2} can also be used to show that the bipartite $K_{2,2}$-free process terminates in $O(n^{4/3} \log^{1/3}n)$ steps, matching the lower bound we proved in this paper.

\bibliographystyle{plain}

\end{document}